\documentclass[11pt,a4paper]{amsart}
\usepackage[utf8]{inputenc}

\usepackage{tikz, float}
\usepackage{comment}
\usepackage{caption}
\usepackage{subcaption}
\usepackage{amsmath,amsthm,amsfonts,graphicx,amssymb,amscd,dsfont,euscript,enumerate,verbatim,calc,mathtools,listings}
\usepackage{mathrsfs}
\usepackage[hidelinks]{hyperref}
\usepackage[nameinlink, noabbrev]{cleveref}
\newtheorem{thm}{Theorem}[section]

\newtheorem{cor}[thm]{Corollary}
\newtheorem{lem}[thm]{Lemma}

\newtheorem{prop}[thm]{Proposition}

\theoremstyle{definition}
\newtheorem{defn}[thm]{Definition}

\newcommand{\pslnr}{\mathrm{PSL}_n(\mathbb{R})}
\newcommand{\pslnc}{\mathrm{PSL}_n(\mathbb{C})}
\newcommand{\slnc}{\mathrm{SL}_n(\mathbb{C})}
\newcommand{\slnr}{\mathrm{SL}_n(\mathbb{R})}

\newcommand{\psltr}{\mathrm{PSL}_3(\mathbb{R})}

\newcommand{\pslr}{\mathrm{PSL}_2(\mathbb{R})}

\newcommand{\De}{D_\epsilon}

\newcommand{\Z}{\mathbb Z}

\newtheoremstyle{named}%
{}{}{\itshape}{}{\bfseries}{.}{.5em}{\thmnote{#3}}
\theoremstyle{named}

\allowdisplaybreaks[1]
\numberwithin{equation}{section}
\begin{document}
\title[Length-convexity and domination for Hitchin representations]{Length-convexity and domination\\ for Hitchin representations}
\author{Subhojoy Gupta}

\address{Department of Mathematics, Indian Institute of Science, Bangalore 560012, India}
\email{subhojoy@iisc.ac.in}

\begin{abstract}
Let $S$ be a closed oriented surface of genus $g\geq 2$. For any $n\geq 3$, let $\text{Hit}_n(S)$ be the Hitchin component and let $\text{Hit}_n(S)^\prime$ be the subset of representations that take values in $\mathrm{PSp}_{2m}(\mathbb{R})$ if $n=2m$ and $\mathrm{PSO}(m,m+1)$ if $n=2m+1$. In this note we prove that for any Hitchin representation $\rho \in \text{Hit}_n(S)$ there is $j \in \text{Hit}_n(S)^\prime$ such that $\rho$ dominates $j$ in the Hilbert length spectrum; this confirms a conjecture of Lee-Zhang and, independently, Tholozan who had proved this for $n=3$. The key step is establishing a convexity of the Hilbert length function in the Bonahon-Dreyer parametrization of $\text{Hit}_n(S)$; this uses a result from the recent work of Barman-Gupta. We also prove these convexity and domination results for translation length functions and their spectra.  
\end{abstract}

\medskip

\maketitle

\section{Introduction}
Let $S$ be a closed, oriented surface of genus $g\geq 2$ and let $\Gamma = \pi_1(S)$ be its fundamental group. For $n\geq 2$ let $\chi_n^\mathbb{R}(S) = \text{Hom}(\Gamma, \pslnr)/\pslnr$ be the $\pslnr$-representation variety and let  $\text{Hit}_n(S)$ denote the Hitchin component. 

\noindent In what follows, we define $$H_n = \begin{cases}
     \mathrm{PSp}_{2m}(\mathbb{R}) & \text{if } n=2m \\
    \mathrm{PSO}(m,m+1)  & \text{if } n=2m+1
\end{cases} $$
and let $\text{Hit}_n(S)^\prime$  be the subset of $\text{Hit}_n(S)$ that comprises Hitchin representations that are conjugate to ones taking values in $H_n$. This subset is the fixed-point set of the contragredient involution on the Hitchin component (see \S4). 

\medskip 
\noindent Given a Hitchin representation $\rho\in \text{Hit}_n(S)$, and  nontrivial $\gamma \in \Gamma$ the \textit{Hilbert length} of the curve $\gamma$ is 
\begin{equation}\label{eq:hl}
\ell_\rho^H(\gamma) = \log \frac{\lambda_n}{\lambda_1}
\end{equation} 
where $0<\lambda_1 <\lambda_2 <\cdots <\lambda_n$ are the eigenvalues of the unique lift of $\rho(\gamma)$ in $\slnr$ with positive eigenvalues (see \cite[Theorem 1.5]{Lab06}). The Hilbert length of the trivial element is defined to be zero.

\medskip 

\noindent We shall prove:

\begin{thm}\label{thm:main} Given any $\rho \in \emph{Hit}_n(S)$, there is a $j \in \emph{Hit}_n(S)^\prime$ that subdominates $\rho$, that is, satisfies $\ell_j^H(\gamma) \leq \ell_\rho^H(\gamma) $ for all $\gamma \in \Gamma$.
    
\end{thm}

This confirms a conjecture of Lee-Zhang and Tholozan (see \cite[Conjecture 3.6]{LZ} and \cite[pg. 1384]{Tholozan}); for $n=3$ this was proved by Tholozan, with a stronger statement involving a \textit{strict} domination -- see \cite[Theorem B]{Tholozan}. In the case of $n=3$, the subgroup $H_3 =  \mathrm{PSO}(1,2) \cong \pslr$ and hence any element $j \in \text{Hit}_3(S)^\prime$ is in fact a \textit{Fuchsian representation}. However for $n>3$, it was shown in \cite{LZ} (by an argument attributed to Labourie) that the subdomination does \textit{not} hold if one requires $j$ to be $n$-Fuchsian.

\medskip

The proof is based on one part of our recent work \cite{BGup2}, that used and developed an algebraic technique for representing the monodromy $\rho(\gamma)$ as the weight matrix of a weighted planar network. That paper concerned certain \textit{complex} deformations of a Hitchin representation in $\text{Hit}_n(S)$, and proved a domination result there. In the case when $n=2$, this had been proved in work of Deroin-Tholozan in \cite{DT}; their method involved harmonic maps and relied on the fact that the corresponding symmetric space is \textit{strictly} negatively curved, which fails in higher rank. In \S2 we shall recall the linear-algebraic setup in \cite{BGup2}, which in turn, relies on the recent work of Maloni-Martone-Mazzoli-Zhang (\cite{MMMZ}) that generalized the Bonahon-Dreyer parametrization of $\text{Hit}_n(S)$ in  \cite{BD14} and \cite{BD17}.

\medskip
\noindent Recall that the Bonahon-Dreyer parametrization in \cite[Theorem 0.4]{BD17} is a homeomorphism  
\begin{equation}\label{eq:bd}
\Phi: \text{Hit}_n(S) \longrightarrow   \mathcal{C}_n(\lambda) \subset \mathbb{R}^N
\end{equation}
where $\mathcal{C}_n(\lambda)$ is a convex polyhedral cone in the vector space of dimension $N:=(n^2-1)(2g-2)$ comprising generalized shear-bend cocycles transverse to a maximal lamination $\lambda$ on $S$.  The key intermediate result,  that could be of independent interest, is:

\begin{thm}\label{thm:convex}
In the linear cone $\mathcal{C}_n(\lambda)$, for any  $\gamma \in \Gamma$ the Hilbert length functional $$L_\gamma: \mathcal{C}_n(\lambda) \to \mathbb{R}_{\geq 0}$$ defined by $L_\gamma(z) := \ell^H_{\Phi^{-1}(z)}(\gamma)$ is convex, that is, 
\begin{equation}\label{eq:cvx} 
    L_\gamma(tx + (1-t)y) \leq tL_\gamma(x) + (1-t)L_\gamma(y)
\end{equation}
for any $x,y \in \mathcal{C}_n(\lambda)$ and $t\in (0,1)$.
\end{thm}

We remark that for shear coordinates on Teichm\"{u}ller space (which can be thought of as the Hitchin component for $n=2)$, it is well-known that length functions  are convex (see \cite{BBFS}); Theorem \ref{thm:convex} is a higher-rank analogue. In the higher-rank case, such a convexity is proved along certain Hamiltonian curves in recent work of Bridgeman-Labourie, see \cite[Theorem B]{BridLab}. 

\medskip

The proof relies on the approximations to the monodromy matrix  that we develop in \cite{BGup2}, in particular a certain positivity the entries have in terms of the parameters. Using the linear-algebraic techniques introduced in that paper and its prequel \cite{BGup}, we can also prove the convexity of \textit{fundamental weight lengths} (see \eqref{eq:fwt}) and consequently Theorems \ref{thm:main} and  \ref{thm:convex} for translation lengths instead of Hilbert lengths. Here, recall that the \textit{translation length} of $\gamma$ is 
\begin{equation}\label{eq:transl}
    \ell^X_\rho(\gamma)= \sqrt{2((\log\lambda_1)^2 + \cdots + (\log\lambda_n)^2)}
\end{equation}
 where, as in \eqref{eq:hl}, $0<\lambda_1 <\cdots <\lambda_n$ are the eigenvalues of a determinant-one representative of $\rho(\gamma)$.

\begin{thm}\label{thm:transl} In the linear cone $\mathcal{C}_n(\lambda)$, for any  $\gamma \in \Gamma$ the translation length functional
    $$L^{X}_\gamma: \mathcal{C}_n(\lambda) \to \mathbb{R}_{\geq 0}$$ defined by $L^{X}_\gamma(z) := \ell^X_{\Phi^{-1}(z)}(\gamma)$ is convex. Moreover, any $\rho \in \text{Hit}_n(S)$ dominates  a representation $j\in \text{Hit}_n(S)^\prime$ in the translation length spectrum. 
\end{thm}

\medskip

The strategy for proving the existence of the subdominating representation $j$  in both Theorem \ref{thm:main} and Theorem \ref{thm:transl} is to define a suitable compact convex subset ${K}_\rho \subset \mathcal{C}_n(\lambda) $ using the convex length functions, that the contragredient involution on $\text{Hit}_n(S)$ preserves;  the fixed-point (guaranteed by Brouwer) then yields the desired subdominating representation.  

\medskip

 In the final section, we consider the case $S$ has punctures, and positive representations to $\psltr$, and prove a subdomination result in the representation variety with prescribed peripheral Hilbert lengths (see Proposition \ref{prop:n3}). Such positive representations are parametrized by Fock-Goncharov coordinates with respect to an ideal triangulation (see \cite{FG} and \cite{FG07}); indeed, the work of Bonahon-Dreyer extended these to the case of a closed surface. In this setting, one can provide an explicit construction of a subdominating Fuchsian representation.

\medskip

 For the shear coordinates in  Teichm\"{u}ller space the convexity property is also inherited by length functions of measured laminations (see \cite{Theret}). In subsequent work we plan to explore analogous generalizations of the Hilbert length functional, in particular to H\"{o}lder geodesic currents (as already introduced in \cite{Dreyer}), and address the question of when a {strict} subdomination holds.

\medskip

\noindent\textbf{Acknowledgements.} The key ideas in this note, specifically the convexity of the log-spectral radius function in Corollary \ref{cor:convex} and the use of the Brouwer fixed-point theorem in the proof of Theorem \ref{thm:main},  are due to ChatGPT 5.6 Sol. These emerged in a conversation initiated by the author with the goal of settling the subdomination conjecture using  techniques developed in the recent joint work with Barman. The ideas were then verified by the author and the paper has been written entirely by the author with no AI-generated text or figures; ChatGPT 6 Astra was used only for a final editorial check (but not the actual editing).  The author takes full responsibility for the content of this article. This work was supported by the DST FIST program - 2021 [TPN - 700661].

\section{Preliminaries}

\subsection{Bonahon-Dreyer parametrization}
We briefly recall the Bonahon-Dreyer parametrization $\Phi$ mentioned in the Introduction, that was developed in \cite{BD14} and \cite{BD17}. Recall that a maximal lamination $\lambda$ on $S$ is a closed subset that is a disjoint union of complete geodesics (with respect to a choice of a hyperbolic metric) such that the complementary regions are ideal triangles; throughout we fix such a  lamination $\lambda$.

The key property of a  Hitchin representation $\rho:\Gamma \to \pslnr$ that we shall use is that there is an associated $\rho$-equivariant limit map $$\xi: \partial_\infty \widetilde{S} \to \mathcal{F}(\mathbb{R}^n)$$ where the target is the space of complete flags in $\mathbb{R}^n$. Moreover, this map satisfies some transversality properties that include: 
\begin{itemize}
    \item for any distinct points $x,y \in \partial_\infty \widetilde{S}$, the flags $\xi(x), \xi(y)$ are transverse, i.e.\ $\xi_i(x) \oplus \xi_{n-i}(y) = \mathbb{R}^n$ for any $1\leq i<n$, 
    \item for any distinct triple $x,y,z\in \partial_\infty\widetilde{S}$ the corresponding triple of flags is in general position, namely we have that $\xi_a(x) \oplus \xi_b(y) \oplus \xi_c(z) = \mathbb{R}^n$ whenever $a,b,c$ are positive integers adding up to $n$, 
\end{itemize}
where $\xi_i(x)$ denotes the $i$-th dimensional subspace of the flag $\xi(x)$. (See, for example, the notion of a \textit{Frenet} curve in \cite{Lab06}.) 

In \cite{FG} Fock-Goncharov introduced real parameters that classified a triple of flags $(F_1,F_2,F_3) \in \mathcal{F}(\mathbb{R}^n)$ in general position up to the action of $\pslnr$; these parameters are indexed by $\mathcal{B}_n = \{(a,b,c)\in\Z_{>0}^3:a+b+c=n\}$.

Given a Hitchin representation $\rho$, this defines, for any complementary triangle $\Delta$ in $\widetilde{S} \setminus \widetilde{\lambda}$, a collection of   $(n-1)(n-2)/2$ \textit{triangle parameters} $\{t_\alpha\ \vert\ \alpha \in   \mathcal{B}_n\} $ associated with the triple of flags associated with its ideal vertices (via the limit map $\xi$).

Moreover, for two adjacent triangles $\Delta_1$ and $\Delta_2$, Fock-Goncharov defined \textit{edge-invariants} $e_i$ indexed by $i \in \mathcal{A}_n = \{(a,b)\in\Z_{>0}^2:a+b=n\}$ associated with the common edge.  The quadruple of flags associated with the ideal vertices of this pair is then classified (up to the action of $\pslnr$) by the triangle invariants of $\Delta_1,\Delta_2$, together with these edge-invariants.

\medskip 

In fact, following \cite{MMMZ} (which in turn follows \cite[Proposition 5.1]{BD17}), one can define edge-invariants $e_i(P,Q)$  for any pair of such complementary triangles (\textit{plaques}), which need not be adjacent, using the \textit{slithering map} 
\begin{equation}\label{eq:slith}
    \Sigma_\rho: \tilde{\Lambda}^2 \to \slnr.
\end{equation}
Here, the domain is the space of pairs of (oriented) leaves of $\widetilde{\lambda}$, and the slithering map satisfies the property that $\Sigma_\rho(g_1,g_2)$ takes the flags associated with the endpoints of $g_2$ to those of $g_1$. For a more precise definition, as well as the existence of such a map, see \cite[Definition 4.11 and Theorem 4.12]{MMMZ}; in their work the target group is $\slnc$, while here we restrict to $\slnr$. 

Indeed, if $g_1,g_2$ are the separating sides of $P,Q$ respectively, then the edge-invariants $e_i(P,Q)$ are those associated with the ``adjacent" pair $P$ and $\Sigma_\rho(g_1,g_2)\cdot Q$ (see \cite[\S4.1.5]{MMMZ}).

The generalized shear-bend cocycle $(\alpha_\rho, \theta_\rho) \in \mathcal{C}_n(\lambda)$ corresponding to the representation $\rho$ is then defined as follows: for any  pair of distinct plaques $P,Q$ and $i\in\mathcal{A}_n$, and for a labeled plaque $x$ and $j\in\mathcal{B}_n$, we have 
\begin{equation}\label{eq:exp-coord}
    e_{i}(P,Q)=\exp\bigl(\alpha_\rho^{i}(P,Q)\bigr) \text{  and  } t_{j}(x)=\exp\bigl(\theta_\rho^{j}(x)\bigr).
\end{equation}
We shall denote these collections of coordinates by $e_\rho(P,Q)$ and $t_\rho(x)$ respectively. The total space of such cocyclic-pairs $(\alpha, \theta)$ is a vector space $\mathcal{Y}_n(\lambda; \mathbb{R})$ of dimension $N:=(n^2-1)(2g-2)$ -- see \cite[Definition 4.2, Proposition 4.13 and \S9.4]{MMMZ}. The above assignment is precisely the Bonahon-Dreyer parametrization  \eqref{eq:bd}, that is,  $\Phi([\rho]) = (\alpha_\rho, \theta_\rho)$. This is an embedding to the subset $\mathcal{C}_n(\lambda)$ that is an open polyhedral cone defined by finitely many linear equations and strict linear inequalities -- see \cite[Propositions 0.5 and 0.6]{BD17}.

\subsection{Monodromy formula}

Given a Hitchin representation $\rho$, there is a convenient formula for the monodromy matrix $\rho(\gamma)$ for $\gamma \in \Gamma$ in terms of the shear-bend cocycle $(\alpha_\rho, \theta_\rho)$ associated with the representation, that we shall describe here.

First, we consider the case of a \textit{punctured} surface, and the Fock-Goncharov coordinates of a positive representation, with respect to an ideal triangulation, which can be thought of as a \textit{finite} maximal lamination $\lambda$, that is, a maximal lamination with finitely many leaves between the punctures. Note that up to the action of $\Gamma$, there are finitely many ideal triangles, which are the complementary plaques in $\widetilde{S} \setminus \tilde{\lambda}$.

Following Fock-Goncharov (\cite{FG}), each (oriented) edge of a plaque $P$ is associated with a projective basis of $\mathbb{R}^n$, namely, the one in which the pair of flags $F^+,F^-$ associated with the two endpoints are \textit{opposite}. Following the convention in \cite{MMMZ}, one can use the flag $F^\prime$ associated with the third vertex of $P$ to determine a \textit{canonical snake basis} $\{b_1,b_2,\ldots, b_n\}$ as follows: $$b_i \in F^-_i \cap F^+_{n-i+1} \text{ and } F^\prime_1 = \langle b_1 + b_2 + \cdots +b_n\rangle.$$

For a curve $\gamma \in \Gamma$ that is transverse to $\lambda$, a determinant-one representative in the projective class of the monodromy $\rho(\gamma)$ can then be expressed as a finite product of ``building blocks":
\begin{equation}\label{eq:monfor}
A = T^{\delta_1} E \cdot T^{\delta_2}E \cdots  T^{\delta_m}E 
\end{equation} 
where each $\delta_i \in \{-1,1\}$.

\begin{figure}
\centering
\includegraphics[scale=.4]{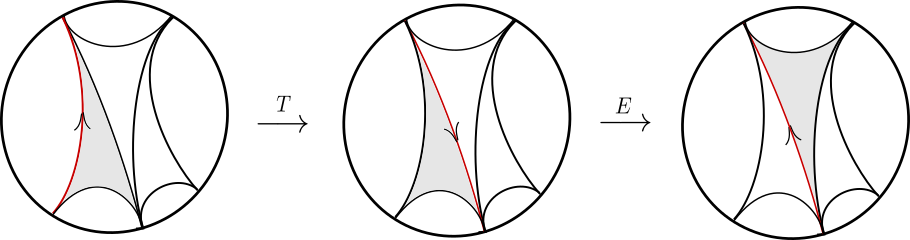}
\caption{The  canonical snake basis is associated with an oriented side (shown in red) of a plaque (shown shaded). The $T$ and $E$ matrices then record the change of basis when the plaque or oriented side change as shown.}
\label{fig:TE}
\end{figure}

Each $T$-matrix is a change-of-basis matrix from one side of a triangular plaque to another, and an $E$-matrix is the change-of-basis matrix that corresponds to choosing the other triangle adjacent to an edge for determining the basis. The above sequence of $T$- and $E$-matrices then corresponds to the plaques and edges intersected by an arc $\tilde{c}_\gamma \subset \widetilde{S}$ that is a  lift of $\gamma$. 

For later use, we note that these matrices can be further expressed as
\begin{equation}\label{eq:te}
    T(t) = M(t) \cdot S \qquad E(e) = D(e)\cdot S
\end{equation}
where $t$ and $e$ are the associated  triangle and edge parameters respectively, $M$ is a matrix whose dependence on the triangle parameters is explicit (see \cite[(3.7)]{BGup}), $D$ is an explicit diagonal matrix (see \cite[(3.8)]{BGup}), and $S$ is the anti-diagonal change-of-basis matrix associated with switching the orientation of an edge, having entries $\bigl(1,-1,1,-1,\ldots,(-1)^{n-1}\bigr)$ along the anti-diagonal and $0$s elsewhere. 

\medskip

\noindent The following notion will be used throughout the article:

\begin{defn}\label{defn:Lau} A \textit{generalized Laurent monomial} in variables $\{x_1,x_2,\ldots x_m\}$ is an expression of the form $$c \cdot x_1^{a_1}x_2^{a_2}\cdots x_m^{a_m}$$
where $a_i \in \mathbb{Q}$ and $c\in \mathbb{R}$. Such a monomial is called \textit{positive} if the scalar coefficient $c>0$.
\end{defn}

It is known that the entries of $M(t), D(e)$, and each ``building block" $T^{\pm 1}E$ are finite sums of positive generalized Laurent monomials in the corresponding triangle and edge parameters -- see \cite[Lemma 2.8]{BGup2} for a summary and \cite{BGup} for details.  Here \textit{rational} (and not just integer) powers of the parameters are needed due to the determinant-one normalization --  see, for example, \eqref{eq:2} and \eqref{eq:3} for the case when $n=3$.  As a consequence, each entry of the monodromy matrix $A$ in \eqref{eq:monfor} is a finite sum of such generalized Laurent monomials with positive coefficients. 

\medskip

Returning to the case of a closed surface with a maximal lamination $\lambda$, the arc $\tilde{c}_\gamma$ could cross infinitely many leaves. To handle this, the work in \cite{MMMZ} (their Theorem 4.12), when restricted to our ``real" setting,  defines the slithering map \eqref{eq:slith} as an infinite product
\begin{equation}\label{decomp-limit} 
\Sigma_\rho(g_1,g_2) = {\mathop{\overrightarrow{\prod}}\limits_{P\in\mathcal{P}(g_1,g_2)}} M(P)
\end{equation} 
where $M: \mathcal{P}(g_1,g_2) \to \slnr $ is a H\"{o}lder-extendable map (in the sense of \cite[Definition 5.1]{MMMZ}) and $\mathcal{P}(g_1,g_2)$ is the  (possibly infinite) set of plaques separating $g_1$ and $g_2$. Here the ordered product is a limit of any finite exhaustion $\mathcal{F}_n \nearrow   \mathcal{P}(g_1,g_2)$, that converges in $\slnr$ and is independent of the exhaustion by \cite[Proposition 5.2]{MMMZ} or \cite[Lemma 5.5]{BD17}.

For the monodromy formula in this case, we use the  work in  \cite[Section 7.4]{MMMZ} that gives an alternative  H\"{o}lder extendable map to define the slithering map.  Namely, in their notation, if $h$ is a leaf of the lamination and $g_x^h$ is the boundary of a plaque $x$, then there is a  H\"{o}lder-extendable map 
\begin{equation}\label{eq:Lfactor}
L_{\rho}:\mathcal{P}(g_x^h,h) \to \slnr
\end{equation}
such that 
\begin{equation}\label{eq:slith2}
\Sigma_\rho(h,g_x^h) = \overrightarrow{\prod\limits_{P}}\ {L_\rho}(P)
\end{equation}
where note that the order of the two arguments have been swapped -- see \cite[Lemma 7.11(2)]{MMMZ}. 

\medskip

As explained in \cite[Lemma 3.12 and 3.13]{BGup2} (see also \cite[Lemma 7.11]{MMMZ}), each $L_\rho(P)$ is a unipotent matrix (see \cite[Figure 4]{BGup2} for a summary) that has an explicit description in terms the shear parameters $e_\rho(x,P)$ and triangle parameters $t_\rho(P)$ in the canonical basis associated with  the leaf $g_x^h$ (considered as a side of $x$). We shall need the following consequence; throughout $D_\epsilon$ is the diagonal matrix with diagonal entries $\bigl(1,-1,1,-1,\ldots,(-1)^{n-1}\bigr)$, $J$ is the anti-diagonal matrix with $1$s along the anti-diagonal (and $0$s elsewhere), and note that $S$ in \eqref{eq:te} equals the product $D_\epsilon J$.

\begin{lem}\label{lem:form0}  The entries of the matrix $D_\epsilon^{-1}L_\rho(P)D_\epsilon$ are finite sums of generalized Laurent monomials in triangle and edge parameters of $\rho$ with positive coefficients. 
\end{lem}

\begin{proof}
    By Lemma 3.13 of \cite{BGup2}, the matrix $D_\epsilon^{-1}L_\rho(P)D_\epsilon$ is a conjugate of $J\widehat{M}J$ or $\widehat{M}$ by diagonal matrices $Q_+$ or $Q_-$ respectively, depending on the cyclic order of the vertices of $P$ relative to $x$. Here, $Q_\pm$ are matrices whose entries are Laurent monomials in edge invariants $e_i(x,P)$  -- we refer to the statement of Lemma 3.13 of that paper for details. Moreover, $\widehat{M}$ is matrix which, by Lemma 3.12 of \cite{BGup2},  has entries that are finite sums of generalized Laurent monomials in the triangle parameters $t(P)$, with positive coefficients.  Thus, each entry of the matrix $D_\epsilon^{-1}L_\rho(P)D_\epsilon$ is then a finite sum of Laurent monomials in these parameters, with positive coefficients, as desired. 
\end{proof}

\noindent As a consequence, we obtain:

\begin{lem}\label{lem:form}  For any non-trivial $\gamma \in \Gamma$ such that its axis is not a leaf of $\tilde{\lambda}$
the monodromy  $\rho(\gamma)$ has, up to conjugation,  a representative $A \in \slnr$ that can be approximated by such matrices  $\{A_N\}_{N\geq 1}$, that is  $A_N \to A$ as $N\to \infty$, such that the entries of each $A_N$ are finite sums of positive generalized Laurent monomials in triangle and edge parameters of $\rho$ (with respect to $\lambda$).  
\end{lem}
\begin{proof}
As in the proof of \cite[Lemma 3.16]{BGup2}, if the lift $\tilde{c}_\gamma \subset \widetilde{S}$ of $\gamma$  starts in the interior of a plaque $P$, crossing its side $g_P$, and terminates in the plaque $Q = \gamma \cdot P$ after crossing the side $g_Q$, then by \cite[Lemma 3.9]{BGup2} the matrix
\begin{equation}\label{eq:closed-monodromy-factorization}
 A
 =T\bigl(t_\rho(Q)\bigr)^\delta
  \Sigma_\rho(g_Q,g_P)
  E\bigl(e_\rho(P,Q)\bigr)
\end{equation}
sends the  canonical snake basis corresponding to $g_P$  to its $\gamma$-translate, and is hence a representative of $\rho(\gamma)$.

We know by \eqref{eq:slith2} and \cite[Proposition 5.2]{MMMZ} that the slithering factor $ \Sigma_\rho(g_Q,g_P)$ admits a sequence of approximations 
\begin{equation}\label{eq:approx}
S_N =  {\prod\limits_{P\in \mathcal{F}_N}} {L_\rho}(P)
\end{equation}
where $\mathcal{F}_N \nearrow \mathcal{P}(g_P,g_Q)$ is an exhaustion by finite sets. Note that by Lemma \ref{lem:form0}, the entries of the conjugate 
\begin{equation}\label{eq:sn}
D_\epsilon^{-1} S_N D_\epsilon = (D_\epsilon^{-1} L_\rho(P_1) D_\epsilon)( D_\epsilon^{-1} L_\rho(P_2) D_\epsilon )\cdots (D_\epsilon^{-1} L_\rho(P_k) D_\epsilon)
\end{equation}
are finite sums of positive generalized Laurent monomials in  triangle and edge parameters of $\rho$. Hence we have the sequence of approximations 
\begin{equation}\label{eq:anapp}
A_N = T\bigl(t_\rho(Q)\bigr)^\delta \cdot 
  S_N\cdot 
  E\bigl(e_\rho(P,Q)\bigr)
  \end{equation}
  
  such that $A_N \to A$.

\medskip

\noindent There are two cases; for both we use \eqref{eq:te} and the identities $S^2= (-1)^{n-1}I, D_\epsilon^2 = J^2 = I$ and $S = D_\epsilon J$:

\begin{itemize}
    \item[(i)] When  $\delta = +1$, we have $$TS_NE = (MD_\epsilon J)S_N (DD_\epsilon J) = M(D_\epsilon J D_\epsilon) (\De^{-1} S_N \De) DJ = (-1)^{n-1} MJ (\De^{-1} S_N \De) DJ$$
where for the second equality  we have used that $D$ and $\De$ commute. Since we are working in a projective class, in the case $n$ is even we define $A_N$ be $-TS_NE$ for it to have non-negative entries. We conclude that all the entries of $A_N$ are finite sums of generalized Laurent monomials in the triangle and edge parameters, since the same is true for each of the factors on the right. 

\item[(ii)] When $\delta = -1$, we can write
$$T^{-1}S_NE = (S^{-1}M^{-1}S) (S^{-1}S_N  S) (S^{-1}DS) = (S^{-1}M^{-1}S)J (\De^{-1} S_N \De)DJ $$
where note that the first factor on the right has entries that are finite sums of generalized Laurent monomials by \cite[proof of Proposition 3.10]{BGup}. Since the same is true for the other factors, once again, the product has entries with the same property.

\end{itemize}

\noindent Thus, in both cases the approximating matrix $A_N$ has the desired property, and we are done.  \end{proof}

\noindent Note that the above construction only depends on the intersection pattern of $\gamma$ with the plaques in the complement of $\tilde{\lambda}$; as one varies $\rho \in \text{Hit}_n(S)$ the decomposition \eqref{eq:anapp} remains the same and only the parameters $t_\rho,e_\rho$ vary. Thus, we in fact obtain the following more uniform statement:

\begin{cor}\label{cor:form} For any $\gamma$ as in the above Lemma, and each $N\geq 1$, there is a matrix-valued function $A_N: \mathcal{C}_n(\lambda) \to \slnr$ whose entries are finite sums of positive generalized Laurent monomials in the triangle and edge parameters, such that $A_N(v) \to A(v)$ as $N\to \infty$, where $A(v)$ is a representative in $\slnr$ conjugate to $\rho(\gamma)$ for $\rho = \Phi^{-1}(v)$.
\end{cor}

\section{Convexity of Hilbert length}

Throughout this section, we fix a maximal lamination $\lambda$ on $S$; as in \S2.1 let $\mathcal{C}_n(\lambda)$ be the convex cone in the vector space $\mathcal{Y}_n(\lambda)$ of shear-bend cocycles parametrizing the Hitchin component. From the discussion in \S2.1, it follows that any triangle parameter $t_a$ (corresponding to some complementary ideal triangle $\Delta$) or edge parameter $e_i$ defines a function on $\mathcal{C}_n(\lambda)$ of the form $\exp(u)$ where $u$ is a linear functional on $\mathcal{Y}_n(\lambda)$. 

\medskip
\noindent The key observation is this basic fact:

\begin{lem}\label{lem:convex} Let $u_1,u_2,\ldots u_m$ be linear functionals on a finite-dimensional vector space $V$, and let $x_i = \exp(u_i)$  for $1\leq i\leq m$. Then the logarithm of a finite sum of generalized Laurent monomials in $\{x_1,x_2,\ldots, x_m\}$ with positive coefficients is a convex function on $V$.
\end{lem}
\begin{proof}
Note that a generalized Laurent monomial in $\{x_1,x_2,\ldots, x_m\}$ is of the form $\exp(u_\alpha)$ for some linear functional $u_\alpha:V \to \mathbb{R}$.  Let  $$P =\sum\limits_\alpha c_\alpha \exp({u_\alpha})$$
be the function on $V$, where each $u_\alpha$ is a linear functional and each $c_\alpha>0$.  Then  we can write
\begin{equation*}
    P((1-t)x + ty) =  \sum\limits_\alpha c_\alpha \exp({(1-t)u_\alpha(x) + t u_\alpha(y)}) =  \sum\limits_\alpha (c_\alpha \exp({u_\alpha(x)}))^{1-t} (c_\alpha \exp({u_\alpha(y)}))^{t} 
\end{equation*}
which by H\"{o}lder's inequality
\begin{equation*}
\leq \left(\sum\limits_\alpha  c_\alpha \exp({u_\alpha(x)}) \right)^{1-t} \left(\sum\limits_\alpha  c_\alpha \exp({u_\alpha(y)}) \right)^{t} = P(x)^{1-t} P(y)^t.
\end{equation*}
Taking a logarithm, we obtain the desired convexity.     
\end{proof}

\medskip

\noindent The following corollary is also implied by the result of Kingman in \cite{Kingman}.

\begin{cor}\label{cor:convex} Let $V$ be a convex subset of a finite-dimensional vector space, and $M:V \to \text{M}_{n\times n}(\mathbb{R})$ be a matrix-valued function such that the matrix entries are finite sums of positive generalized Laurent monomials in  $\{x_1,x_2,\ldots, x_m\}$ where $x_i = \exp(u_i)$ and $u_i$ are linear functionals on $V$.  Then if the spectral radius $\lambda_n(M(v))>0$ for any $v\in V$,
the log-spectral radius function $$v \mapsto \log(\varrho(M(v)))$$ is convex.
\end{cor}

\begin{proof}
By our assumption the entries of $M(v)$ are non-negative for any $v\in V$. In particular, the norm $\lVert A\rVert_s := \sum\limits_{i,j} \lvert A_{ij}\rvert =  \sum\limits_{i,j} A_{ij}$ is the $l_1$ matrix norm and the log-spectral radius 
$$\log(\varrho(M(v))) = \lim\limits_{k\to \infty} \frac{1}{k}\log (\lVert M(v)^k \rVert_s)$$
by Gelfand's formula (see for example, \cite[Corollary 5.6.14]{HJ}). 

For any fixed integer $k\geq 1$, each entry of the matrix power $M(v)^k$ is also a finite sum of generalized Laurent monomials in  $\{x_1,x_2,\ldots, x_m\}$ with positive coefficients. Hence by the previous Lemma, each term in the right-hand side is convex, and so is the limit as $k\to \infty$.    
\end{proof}

\noindent The convexity of Hilbert lengths (Theorem \ref{thm:convex}) then follows from the observation about the monodromy formula in \S2.2:

\begin{proof}[Proof of Theorem \ref{thm:convex}]

Let $\gamma \in \Gamma$; assume first that $\gamma$ is not a component of the lamination $\lambda$ or a power of such a closed curve (equivalently, its axis is not a leaf of $\tilde{\lambda}$). For any $v\in \mathcal{C}_n(\lambda)$ that corresponds to $\rho_v := \Phi^{-1}(v) \in \text{Hit}_n(S)$, there is a representative $A_\gamma(v) \in \slnr$ in the  projective class of the monodromy matrix $\rho_v(\gamma)\in \pslnr$  such that $A_N(v) \to A_\gamma(v)$ as $N\to \infty$  where by Corollary \ref{cor:form}  each $A_N(v)$ satisfies the hypotheses of Corollary \ref{cor:convex}. Note that each triangle and edge parameter is of the form $\exp(u)$ for some linear functional $u$ by \eqref{eq:exp-coord}.  

Hence, $\log \varrho(A_N(v))$ is convex; by the continuity of spectral radius, so is the limit $\log \varrho(A_\gamma(v))$ as $N\to \infty$. The same argument applied to the inverse curve $\gamma^{-1}$ then shows that $\log \varrho(A_{\gamma^{-1}}(v))$ is also convex.
Since the spectral radius $\varrho(A_{\gamma^{-1}}(v))$ is the reciprocal of the smallest  eigenvalue (i.e. $1/\lambda_1$) of  $\rho_v(\gamma)$, the sum $$L_\gamma(v) =  \log \varrho(A_\gamma(v)) + \log \varrho(A_{\gamma^{-1}}(v))$$ is precisely the Hilbert length function of the curve $\gamma$ for the Hitchin representation $\rho_v$.  
The sum of two convex functions is also convex, and this proves \eqref{eq:cvx}.

It remains to handle the case when $\gamma$ coincides with a component of $\lambda$ (or a power); here we can apply \cite[Theorem 7.5]{BD17} to show that the length function $L_\gamma(v)$ is in fact \textit{linear} along any line in the cone $\mathcal{C}_n(\lambda)$. Indeed, from \cite{BD17} we know that the vector space of generalized shear-bend cocycles can be interpreted as a relative homology of a train-track neighborhood; their Theorem 7.5 then asserts that the $i$-th length function $\ell^i_{\rho_v}(\gamma) = \log \lvert \lambda_{i+1}/\lambda_i\rvert$ (where $1\leq i\leq n-1$)  equals the intersection number of $\gamma$ with the $\alpha$-component of the shear-bend cocycle corresponding to $v$, and is hence linear as a function of $v$. Since the Hilbert length $L_\gamma$ is the sum of these $(n-1)$ length functions, it is also linear. 
\end{proof}

\section{Proof of Theorem \ref{thm:main}}

Let $\mathcal{C}_n(\lambda)$ be the convex cone identified with the Hitchin component $\text{Hit}_n(S)$ via the Bonahon-Dreyer parametrization; let $\rho$ be a Hitchin representation.

\medskip

\noindent Let 
\begin{equation}\label{eq:inv}
\sigma: \mathcal{C}_n(\lambda) \to \mathcal{C}_n(\lambda)
\end{equation}
be the involution induced by the contragredient involution on $\text{Hit}_n(S)$ defined by $[\eta] \mapsto [\eta^\ast]$ where $[\eta^\ast]$ is the (conjugacy class) of the representation defined by the inverse-transpose $\eta^\ast(\gamma) = (\eta(\gamma)^{-1})^T$ for any $\gamma \in \Gamma$. 

\noindent The following facts are well-known (see, for example, \cite[\S2.5]{BCL}):
\begin{itemize}
    \item[(a)] The fixed-point set of $\sigma$ is precisely the set of representations that correspond to  $\text{Hit}_n(S)^\prime$, that is, taking values in the subgroup $H_n \subset \pslnr$ as defined in the Introduction.  
    \item[(b)] The Hilbert length of a curve  is invariant under the involution, namely $\ell^H_\rho(\gamma) = \ell^H_{\rho^\ast}(\gamma)$ for any $\gamma \in \Gamma$. Indeed, the eigenvalues of $\rho^\ast(\gamma)$ are $\lambda_n^{-1}< \lambda_{n-1}^{-1} <\cdots <\lambda_1^{-1}$ -- \textit{cf.} \eqref{eq:hl}.
\end{itemize}
In what follows we shall refer to these as ``property (a)" and ``property (b)".

\medskip

\begin{lem}\label{lem:setK}
    The subset 
    \begin{equation}\label{eq:Kx0}
    K_\rho =  \{ \eta \in \text{Hit}_n(S)\ \vert \ell^H_\eta(\gamma) \leq \ell^H_\rho(\gamma) \text{ for all } \gamma \in \Gamma\}
     \end{equation}
    is a compact convex set in $\mathcal{C}_n(\lambda)$ that is invariant under the contragredient involution, that is, $\sigma(K_\rho) \subset K_\rho$.
\end{lem}
\begin{proof}
    The invariance under $\sigma$ follows immediately from the invariance of Hilbert lengths in property (b) noted above. 
    For any $\gamma \in \Gamma$, consider the subset $$K_\rho(\gamma) =   \{ \eta \in \text{Hit}_n(S)\ \vert \ell^H_\eta(\gamma) \leq \ell^H_\rho(\gamma) \}$$
    Since  $L_\gamma$ is a convex function by Theorem \ref{thm:convex}, this sub-level set is a convex set, and so is the intersection $$K_\rho = \bigcap\limits_{\gamma} K_\rho(\gamma).$$ 
    The compactness of $K_\rho$ follows from \cite[Corollary 1.5]{LZ} that shows the properness of the map $$\eta \mapsto (\ell^H_\eta(\gamma_1), \ell^H_\eta(\gamma_2), \ldots, \ell^H_\eta(\gamma_k)) $$
    where $\{\gamma_1,\gamma_2,\ldots, \gamma_k\}$ is a filling collection of curves containing a pants decomposition. Indeed, if we set $B = \max\limits_{1\leq i\leq k} \ell^H_\rho(\gamma_i)$, we have that $\eta \in K_\rho \implies  \ell^H_\eta(\gamma_i) \leq B$ for each $1\leq i\leq k$, and hence $K_\rho$ is a (closed) subset of the preimage of the compact set $[0,B]^k$. 
\end{proof}

\medskip

\noindent We can now complete:

\begin{proof}[Proof of Theorem \ref{thm:main}] 

Since $K_\rho$ is a compact convex set in Euclidean space by Lemma \ref{lem:setK}, nonempty as it contains $\rho$, and the contragredient involution $\sigma:K_\rho \to K_\rho$  is continuous, the Brouwer fixed-point theorem applies, and there is a fixed point $j\in K_\rho$ such that $\sigma(j) =j$. By property (a) of this involution stated above, this fixed-point $j$ corresponds to a representation in $\text{Hit}_n(S)^\prime$. By the definition of $K_\rho$ (see \eqref{eq:Kx0}) we have that $\ell^H_j(\gamma) \leq \ell^H_\rho(\gamma)$ for all $\gamma \in \Gamma$, as desired. \end{proof}

\section{Proof of Theorem \ref{thm:transl}}

We recall from \cite[Definition 1.1]{BGup2} the definition of the $k$-th \textit{fundamental weight length}:
\begin{equation}\label{eq:fwt}
\ell^{\omega_k}_\rho(\gamma) = \sum\limits_{i=0}^{k-1}\log\lambda_{n-i}
\end{equation} 
where $\gamma \in \Gamma$ and $\rho\in \text{Hit}_n(S)$, and as before $0<\lambda_1< \dots <\lambda_n$ are the eigenvalues of a determinant-one representative of $\rho(\gamma)$. The convexity of the associated length function, for each $1\leq k<n$, will be our first intermediate result. Note that since \cite{BGup2} concerned representations to $\pslnc$, the eigenvalues were possibly complex, which necessitated taking their moduli; we do not need this in our current setting for $\rho\in \text{Hit}_n(S)$. 

\medskip
For this, we also recall from \cite[Definition 2.6]{BGup2} the notion of a weighted planar network: a planar acyclic directed graph with $n$ sources and $n$ sinks, and a weight associated with each edge that is (in our case) a positive real number. Such a planar network also has an associated $n\times n$ weight matrix, where the $ij$-th entry is the sum of the weights of paths from source-$i$ to sink-$j$ and the weight of a path is the product of the weights along the edges. (For an example, see \cite[Example 5]{BGup}.) 

The key fact about weighted planar networks and their weight matrices, that we shall need in this section, is Lindstr\"{o}m's lemma which expresses minors of the weight matrix also as sums of families of paths in the planar network -- see \cite[Lemma 2.7]{BGup2} for an exact statement, and \cite{Lind} for the exact reference. 

\medskip

\noindent The following is a consequence of the work in \cite[\S3.2]{BGup} and \cite[\S3]{BGup2}:

\begin{lem}\label{lem:anwt} For $v\in \mathcal{C}_n(\lambda)$ and any non-trivial $\gamma\in \Gamma$ such that its axis is not a leaf of $\tilde{\lambda}$ the approximating matrix $A_N(v)$ for $\rho_v(\gamma)$ as in Lemma \ref{lem:form} is the weight matrix of a weighted planar network, where each weight is a  positive generalized Laurent monomial in the triangle and edge parameters of $\rho$. 
\end{lem}

\begin{proof} 
    The idea is to use the decomposition \eqref{eq:anapp} and the fact that the weight matrix of the concatenation of two weighted planar networks is the product of their respective weight matrices. 

    We begin with the observation that by \cite[Lemma 3.13]{BGup2}, for each factor $L_\rho(P)$ in \eqref{eq:approx}, the matrix $D_\epsilon^{-1}L_\rho(P)D_\epsilon$ is a conjugate of $J\widehat{M}J$ or $\widehat{M}$ by the diagonal matrices $Q_+$ or $Q_-$. 
    
    \noindent As in case (i) of the proof of Lemma \ref{lem:form}, if $\delta =+1$ in \ref{eq:anapp}, we can write $$A_N = MJ(D_\epsilon^{-1}S_ND_\epsilon)DJ.$$
    Now by \eqref{eq:sn} and \cite[Lemma 3.13]{BGup2} we have that the factor $(D_\epsilon^{-1}S_ND_\epsilon)$ is the weight  matrix of a planar network with weights of the desired form. Now $D$ is a diagonal matrix with weights that are generalized Laurent monomials in edge parameters (see \cite[(3.8)]{BGup}); this is the weight matrix of a planar network with exactly $n$ edges, from source-$i$ to sink-$i$ for each $1\leq i\leq n$, with the diagonal entries as the edge-weights. The same property is true for the conjugate  $J(D_\epsilon^{-1}S_ND_\epsilon)DJ$  by  \cite[Lemma 3.11 (4)]{BGup2}, and finally for the remaining factor $M$ by \cite[Lemma 3.4]{BGup}.

    \noindent If $\delta =-1$ in \ref{eq:anapp}, as in case (ii) of the proof of Lemma \ref{lem:form}, we have $$A_N = (S^{-1}M^{-1}S)J (\De^{-1} S_N \De)DJ  .$$

    The first factor on the right $(S^{-1}M^{-1}S)$ is the weight matrix of a planar network with weights of the desired form by the work in \cite[\S3.2]{BGup} -- see the proof of \cite[Proposition 3.10]{BGup}. The remaining product is exactly the same as in the discussion above, and we are done.  
\end{proof}

\begin{cor}\label{cor:anwt} For $v, \gamma$ as in the Lemma above, the entries of the $k$-th exterior power  $\bigwedge^kA_N(v)$ are finite sums of positive generalized Laurent monomials in the triangle and edge parameters of $\rho$. In particular, each fundamental weight length functional $$L^{\omega_k}_\gamma: \mathcal{C}_n(\lambda) \to \mathbb{R}_{\geq 0}$$ defined by $L^{\omega_k}_\gamma(v) = \ell^{\omega_k}_{\Phi^{-1}(v)}(\gamma)$ is convex, for any $\gamma \in \Gamma$. 
\end{cor}

\begin{proof}
    Consider first the case when $\gamma \in \Gamma$ is not a closed leaf of $\lambda$ (or a power). Then Lemma \ref{lem:anwt} applies and  $A_N(v)$ is the weight matrix of a weighted planar network with weights that are positive generalized Laurent monomials in the triangle and edge parameters of $\rho_v = \Phi^{-1}(v)$. By Lindstr\"{o}m's lemma, any $k\times k$ minor of $A_N(v)$  is the sum of weights of certain path-families in this planar network, and is hence a finite sum of such monomials.  Such minors are exactly the entries of the exterior power $\bigwedge^k A_N(v)$. 
    
    \noindent  By Corollary \ref{cor:convex}, the logarithm of the spectral radius of this $k$-th exterior power of $A_N(v)$ is convex as a function of $v$.  By continuity, since $A_N(v) \to A(v)$ as $N\to \infty$ (see Corollary \ref{cor:form}), the limit, which is
      $$\log \varrho\left(\bigwedge^k A(v)\right) = \log \lambda_{n-k+1}(v) \cdots \lambda_{n-1}(v)\lambda_n(v),$$
      the product of the $k$ largest eigenvalues of $A(v)$, is also convex. This is exactly the $k$-th fundamental weight length functional  $L^{\omega_k}_\gamma$ -- see \eqref{eq:fwt}. 

In the case that $\gamma \in \Gamma$ is a power of a closed component of $\lambda$, then as observed in the final part of the proof of Theorem \ref{thm:convex}, the $i$-th length function $\log (\lambda_{i+1}/ \lambda_i)$ is linear by \cite[Theorem 7.5]{BD17}. The determinant-one normalization implies that $\log\lambda_1 + \log\lambda_2 + \cdots + \log\lambda_{n}=0$; using this one can express each fundamental weight length as a linear combination of these length functions, and the corresponding length functional $L^{\omega_k}_\gamma$ is linear.
\end{proof}

\noindent We can now complete the proof of Theorem \ref{thm:transl}:

\begin{proof}[Proof of Theorem \ref{thm:transl}]
First, the convexity of the translational length functional $L^X_\gamma$ uses standard majorization inequalities as in \cite[proof of Proposition 3.6]{BGup2}, see also \cite[\S4.2]{BGup}, as we shall now describe. 

\noindent In what follows, for any $v\in \mathcal{C}_n(\lambda)$ let $$a(v) = (\log\lambda_1(v), \log\lambda_2(v),\ldots, \log\lambda_n(v))$$
where $0<\lambda_1(v)<\cdots< \lambda_n(v)$ are the eigenvalues of a determinant-one representative of $\rho_v(\gamma)$; as before $\rho_v = \Phi^{-1}(v) \in \text{Hit}_n(S)$ is the Hitchin representation corresponding to the point $v$. 

Let $x,y\in \mathcal{C}_n(\lambda)$, $t\in (0,1)$ and $\gamma \in \Gamma$; then the convexity of the fundamental weight length functions in Corollary \ref{cor:anwt} implies $$a(tx+ (1-t)y) \prec ta(x) + (1-t)a(y)$$
where the majorization of two $n$-tuples $(a_1,a_2,\ldots,a_n) \prec (b_1,b_2,\ldots, b_n)$ means $$\sum\limits_{i=1}^k a_{n-i+1} \leq \sum\limits_{i=1}^k b_{n-i+1} \text{ for each }1\leq k<n \text{ and }\sum\limits_{i=1}^n a_{i} = \sum\limits_{i=1}^n b_{i}.$$

If $\lVert (a_1,a_2, \ldots, a_n) \rVert_2: = \sqrt{a_1^2+a_2^2+\cdots +a_n^2}$ is the usual convex $l^2$-norm, an application of Karamata's inequality (see \cite[Proposition 3.C.1]{MOA}) then implies the first inequality here: 
$$\lVert a(tx +(1-t)y)\rVert_2 \leq \lVert ta(x) + (1-t) a(y)\rVert_2 \leq t\lVert a(x)\rVert_2 + (1-t) \lVert a(y)\rVert_2$$
where the second is the usual triangle inequality.
Multiplying both sides by $\sqrt{2}$, and using \eqref{eq:transl} we conclude that 
$$L^X_\gamma(tx+ (1-t)y) \leq tL^X_\gamma(x) + (1-t)L^X_\gamma(y) $$
which is the desired convexity.

\smallskip

Now given a $\rho \in \text{Hit}_n(S)$, the proof of the subdominating representation $j$  (in the translation length spectrum) follows the argument in \S4: first, the subset 
$$K^X_\rho = \{ \eta \in \text{Hit}_n(S)\ \vert \ell^X_\eta(\gamma) \leq \ell^X_\rho(\gamma) \text{ for all } \gamma \in \Gamma\} $$
is a convex subset of the Hitchin cone $\mathcal{C}_n(\lambda)$ by the convexity of each $L^X_\gamma$ proved above, and is non-empty as it contains $\rho$. 
Moreover, since 
$$(\ell^H_\eta(\gamma))^2 =  (\log\lambda_n - \log \lambda_1)^2 \leq  2((\log\lambda_n)^2 + (\log\lambda_{n-1})^2 + \cdots +(\log\lambda_1)^2 ) = (\ell^X_\eta(\gamma))^2 $$
it follows that the properness argument as in the proof of Lemma \ref{lem:setK} implies that $K_\rho^X$ is compact. 

The contragredient involution \eqref{eq:inv} does not affect translation lengths, namely $\ell^X_\rho(\gamma) =\ell^X_{\rho^\ast}(\gamma)$ -- see property (b) after \eqref{eq:inv}. Hence we have $\sigma(K^X_\rho) \subset K^X_\rho$, and by the Brouwer fixed-point theorem there is a fixed point $\sigma(j) =j$ for some $j\in \text{Hit}_n(S)^\prime$ (see property (a) in \S4). From the definition of $K^X_\rho$ above, it follows that $\ell^X_j(\gamma) \leq \ell^X_\rho(\gamma)$ as desired.  Note that we can in fact require that the same $j$ is also subdominating in the Hilbert-length spectrum, since we can apply the above fixed-point argument to the restriction of $\sigma$ to the intersection $K_\rho\cap K^X_\rho$ of the two compact convex sets. 
\end{proof}

\section{Positive representations into $\psltr$}

In this section, we consider $S$ to be a surface with punctures, and we consider positive framed representations  $\rho:\pi_1(S) \to \psltr$ where the framing is a choice of a flag in $\mathcal{F}(\mathbb{R}^3)$ at each puncture that is invariant under the peripheral holonomy. Such a framed representation admits Fock-Goncharov coordinates with respect to a choice of an ideal triangulation (namely, one triangle parameter associated with each triangle and two edge parameters associated with each edge, as mentioned in \S2.1). The underlying representation being positive corresponds exactly to these coordinates being positive (see, for example, \cite[Theorem 2.5]{FG07}); they arise as the holonomies of marked convex projective structures on $S$ (see \cite[Theorem 2.7]{FG07}, or \cite[Chapter 4]{CTT}). If $ \mathcal{T}_3^+(S)$ is the space of positive framed representations of $\pi_1(S) \to \psltr$ up to conjugation,  one can define, analogous to \eqref{eq:bd}, a parametrization 
\begin{equation}\label{eq:psi}
\Psi:\mathcal{T}_3^+(S) \to \mathbb{R}^{T+ 2E} 
\end{equation} 
where $T$ (resp. $E$) is the number of triangles (resp. edges) in the ideal triangulation, that maps $\rho$ to the \textit{logarithms} of the Fock-Goncharov parameters. See, for example, \cite[Theorem 2.5]{FG07} or \cite[Theorem 4.4.1]{CTT} for the ``exponentiated" version of this parametrization. 

\medskip

We shall prove the following analogue of Tholozan's subdomination result in this setting; this was raised as a possibility at the conclusion of \cite[\S1]{BGup}. Although the proof shall follow that of Theorem \ref{thm:main}, one advantage is that the construction of the subdominating Fuchsian representation is explicit, and does not depend on a fixed-point theorem.  In the following statement, a \textit{cusped} representation $\rho$ is one where $\rho(\gamma)$ is regular unipotent (i.e., all eigenvalues $1$ and with a single Jordan block) for every peripheral curve $\gamma$ (\textit{cf.} \cite[Theorem 1.4 (2)]{CZZ}). Also, if $\iota:\pslr\to\psltr$ is the standard irreducible embedding, we shall use $j$  and  $\iota\circ j$ interchangeably.

\begin{prop}\label{prop:n3} Let $S$ be a surface of genus $g$ and $n\ge 1$ punctures  such that $2g-2+n >0$.
Given a positive representation $\rho:\pi_1(S) \to \psltr$, there exists a Fuchsian representation $j:\pi_1(S) \to \pslr$ such that $\ell_j^H(\gamma) \leq \ell^H_\rho(\gamma)$
for all $\gamma \in \pi_1(S)$. Moreover, equality holds for any peripheral curve $\gamma$; in particular, if $\rho$ is cusped, then $j$ is the holonomy of a complete finite-area hyperbolic metric on $S$. 
\end{prop}

(Here the Hilbert length $\ell^H_j$ for a Fuchsian representation $j$ is twice the usual hyperbolic length of $\gamma$, or translation length of $j(\gamma)$.)

\begin{proof} 
In the punctured surface setting, we need the monodromy formula \eqref{eq:monfor} for the Fock-Goncharov coordinates, and do not need subsequent discussion in \S2.2 regarding the slithering map approximants. 
By \cite[(3.2) and (3.3)]{BGup} the ``building blocks"  have the form
\begin{equation}
	 T(z)E(x,y)= \sqrt[3]{\frac{1}{xy^2z}}
	 \begin{pmatrix}
	 x y &0 & 0 \\
	 x y & y & 0 \\
	 x y & y(1+z) & z
	 \end{pmatrix}\label{eq:2}
 \end{equation}
 \begin{equation}
	 T(z)^{-1}E(x,y)=\sqrt[3]{\frac{z}{xy^2}} 
	 \begin{pmatrix}
	 \frac{x y}{z} & \frac{y(1+z)}{z} & 1 \\
	 0 & y & 1 \\
	 0 & 0 & 1
	 \end{pmatrix} \label{eq:3}
 \end{equation}
where $x,y$ are the associated edge parameters and $z$ is the triangle parameter; note that the entries are indeed finite sums of generalized Laurent monomials in the sense of Definition \ref{defn:Lau}. In particular, by the proof of Theorem \ref{thm:convex}, the Hilbert length function $L_\gamma$ is convex as a function of the logarithms of the Fock-Goncharov coordinates. 

\smallskip

In \cite{FG07} Fock-Goncharov provide a formula for the involution $\sigma$: namely, for two adjacent triangles with triangle parameters $X,Y$ respectively, and edge parameters $Z,W$ for the common edge, where the parameters are for a (conjugacy class) of a positive representation $[\eta]$, then the corresponding parameters for  $[\eta^\ast]$ are:
\begin{align} 
X^\prime = 1/X\  &\  Y^\prime = 1/Y \label{eq:inv3}\\
Z^\prime = \frac{W(1+Y)}{Y(1+X)}\ &\  W^\prime = \frac{Z(1+X)}{X(1+Y)}. \label{eq:inv4}
\end{align}
(See \cite[Figure 5]{FG07}.)
The fixed-point set of $\sigma$ comprises (positively framed) Fuchsian representations,  where all the triangle parameters are $1$ and the two edge-parameters for any edge are equal (see also \cite[Lemma 6.2.1]{CTT}).

\smallskip
One can define a convex set $K_\rho$ as in \eqref{eq:Kx0}, that is invariant under $\sigma$. However, in this case, one can determine the fixed-point $j$ more explicitly, in two steps:
\begin{itemize}
    \item[(1)] In the first step, define $\rho_1$ to be the representation whose Fock-Goncharov coordinates are the \textit{geometric mean} of the corresponding coordinates of $[\rho]$ and $[\rho^\ast]$. 
    
 \item[(2)] Then, we define $j$ to be the representation whose Fock-Goncharov coordinates are the \textit{geometric mean} of the corresponding coordinates of $[\rho_1]$ and $[\rho_1^\ast]$. 
    
\end{itemize}

Note that  by \eqref{eq:inv3} the triangle parameters of $\rho_1$ obtained in step (1) are all equal to $1$. By \eqref{eq:inv4} this implies that for $[\rho_1]$, the edge parameters $(Z,W)$ of each edge get swapped by the involution $\sigma$. Hence after step (2), the two edge parameters of each edge are equal to each other, and $j$ lies in the Fuchsian locus.

\smallskip 
Observe that taking the geometric mean of the Fock-Goncharov coordinates of $[\rho]$ and $[\rho^\ast]$ corresponds to taking the midpoint of the straight-line segment between them, in $\mathbb{R}^{T+2E}$ defined by the \textit{logarithms} of these parameters, as in \eqref{eq:psi}.  The convexity of the Hilbert length function $L_\gamma$ then implies that $$\ell^H_{\rho_1}(\gamma)\leq \frac{1}{2} ( \ell^H_{\rho}(\gamma) + \ell^H_{\rho^\ast}(\gamma)) =   \ell^H_{\rho}(\gamma) $$
for all $\gamma \in \pi_1(S)$.
Similarly, $[j]$ corresponds to the midpoint of the line segment between $[\rho_1]$ and $[\rho_1^\ast]$ and the convexity yields 
$$\ell^H_{j}(\gamma)\leq \frac{1}{2} ( \ell^H_{\rho_1}(\gamma) + \ell^H_{\rho_1^\ast}(\gamma)) =   \ell^H_{\rho_1}(\gamma).$$
Combining these we obtain the desired subdomination $\ell^H_{j}(\gamma)\leq \ell^H_{\rho}(\gamma)$ for all $\gamma \in \pi_1(S)$.

\smallskip

For the final statement about the peripheral monodromy, note that as mentioned in the proof of \cite[Proposition 3.3]{BGup}, by \cite[Theorem 5.2.4]{CTT} any such peripheral monodromy is conjugate to a  product of building blocks as in \eqref{eq:monfor} with each $\delta_i=1$ (in which case the resulting product is lower-triangular -- see \eqref{eq:2})  or each $\delta_i=-1$ (in which case it is upper-triangular -- see \eqref{eq:3}). The eigenvalues of the monodromy matrix are then precisely the diagonal elements; by \cite[Theorem 4.5.1]{CTT} we can assume that we have fixed the positive framing (and a peripheral generator) at each puncture such that the diagonal elements are in non-decreasing order, and by our determinant-one normalization they have product equal to $1$.  Now each diagonal element is the product of the corresponding diagonal elements of each of the triangular building blocks, and hence the Hilbert length of the peripheral element is defined by linear equations in the logarithms of the parameters. Since these lengths do not change under a contragredient involution by property (b), the midpoints in steps (1) and (2) above also have the same peripheral lengths.  Indeed, if the diagonal entries of the triangular matrix conjugate to $\rho(\gamma)$ are $\lambda_1 \leq \lambda_2\leq \lambda_3$ and those of $\rho^\ast(\gamma)$ are $\lambda_3^{-1} \leq \lambda_2^{-1}\leq \lambda_1^{-1}$, then the midpoint of the segment between them (in the logarithmic Fock-Goncharov coordinates) corresponds to a triangular matrix with diagonal entries $(\lambda_1/\lambda_3)^{1/2} \leq 1 \leq (\lambda_3/\lambda_1)^{1/2}$, that has the same Hilbert length by \eqref{eq:hl}. 
\smallskip

In the case that $\rho$ is cusped, each peripheral holonomy  is unipotent and all the eigenvalues (which are exactly the diagonal elements in the triangular matrix in the conjugacy class) are equal to $1$, and the Hilbert length  $ \ell^H_\rho(\gamma)=0$  for any peripheral curve $\gamma$. The subdomination proved above  implies  that $\ell^H_j(\gamma)=0$; since $j$ is  Fuchsian (discrete and faithful), this  holds if and only if $j(\gamma)$ is parabolic. \end{proof}

 \bibliographystyle{alpha}
 \bibliography{references}
 
\end{document}